\documentclass[12pt]{amsart}
 
\usepackage{amssymb, amscd, txfonts}
\usepackage{graphicx}
\usepackage[all]{xy} 
\usepackage{array, booktabs}
\usepackage{tabularx} 

\numberwithin{equation}{section}

\newtheorem{theorem}{Theorem}[section]
\newtheorem{proposition}[theorem]{Proposition}
\newtheorem{lemma}[theorem]{Lemma}
\newtheorem{corollary}[theorem]{Corollary}

\theoremstyle{definition}

\theoremstyle{remark}
\newtheorem{remark}[theorem]{Remark}
\newtheorem{claim}[theorem]{Claim}

\renewcommand{\ker}{\operatorname{Ker}}

\newcommand{\Z}{\mathbb{Z}}
\newcommand{\Q}{\mathbb{Q}}
\newcommand{\R}{\mathbb{R}}
\newcommand{\C}{\mathbb{C}}
\newcommand{\proj}{{\mathbb P}}

\newcommand{\Fn}{\mathcal{F}_{n}}
\newcommand{\Fncpt}{\bar{\mathcal{F}}_{n}}

\newcommand{\OL}{{\rm O}^{+}(L)}
\newcommand{\Ost}{\tilde{{\rm O}}^{+}(L)}
\newcommand{\GD}{\Gamma \backslash \mathcal{D}}
\newcommand{\DL}{\mathcal{D}_{L}}

\newcommand{\G}{\mathbb{G}(1, 5)}

\begin{document}

\title[]{Kodaira dimension of universal holomorphic symplectic varieties}
\author[]{Shouhei Ma}
\thanks{Supported by JSPS KAKENHI 15H05738 and 17K14158.} 
\address{Department~of~Mathematics, Tokyo~Institute~of~Technology, Tokyo 152-8551, Japan}
\email{ma@math.titech.ac.jp}
\subjclass[2010]{}

\begin{abstract}
We prove that the Kodaira dimension of the $n$-fold universal family 
of lattice-polarized holomorphic symplectic varieties with 
dominant and generically finite period map 
stabilizes to the moduli number when $n$ is sufficiently large. 
Then we study the transition of Kodaira dimension explicitly, 
from negative to nonnegative, 
for known explicit families of polarized symplectic varieties. 
In particular, we determine the exact transition point 
in the cases of Beauville-Donagi and Debarre-Voisin, 
where the Borcherds $\Phi_{12}$ form plays a crucial role. 
\end{abstract} 

\maketitle

\section{Introduction}\label{sec:intro}

The discovery of Beauville-Donagi \cite{BD} that 
the Fano variety of lines on a smooth cubic fourfold  
is a holomorphic symplectic variety deformation equivalent to 
the Hilbert squares of $K3$ surfaces of genus $8$  
was the first example of explicit geometric construction of polarized holomorphic symplectic varieties. 
Gradually further examples, 
all deformation equivalent to Hilbert schemes of $K3$ surfaces, 
have been found by 
\begin{itemize}
\item Iliev-Ranestad \cite{IR} as the varieties of power sums of cubic fourfolds, 
\item O'Grady \cite{OG} as the double EPW sextics, 
\item Debarre-Voisin \cite{DV} as the zero loci of sections of a vector bundle on the Grassmannian $G(6, 10)$, 
\item Lehn-Lehn-Sorger-van~Straten \cite{LLSS} using the spaces of 
twisted cubics on cubic fourfolds, and more recently 
\item Iliev-Kapustka-Kapustka-Ranestad \cite{IKKR} as the double EPW cubes. 
\end{itemize}

The moduli spaces $\mathcal{M}$ of these polarized symplectic varieties are 
unirational by construction. 
However, if we consider 
the $n$-fold fiber product ${\Fn}\to\mathcal{M}$ of the universal family $\mathcal{F}\to\mathcal{M}$ 
(more or less the moduli space of the varieties with $n$ marked points or its double cover), 
its Kodaira dimension $\kappa({\Fn})$ is nondecreasing with respect to $n$ (\cite{Ka}), 
and bounded by $\dim \mathcal{M}=20$ (\cite{Ii}). 
The main purpose of this paper is to study the transition of $\kappa({\Fn})$ as $n$ grows, 
especially from $\kappa=-\infty$ to $\kappa\geq 0$, 
by using modular forms on the period domain. 
Moreover, we prove that $\kappa({\Fn})$ stabilizes to $\dim \mathcal{M}$ 
at large $n$ for more general families of lattice-polarized holomorphic symplectic varieties. 

Our main result is summarized as follows. 

\begin{theorem}\label{thm:main}
Let ${\Fn}$ be the $n$-fold universal family of 
polarized holomorphic symplectic varieties of 
Beauville-Donagi or Debarre-Voisin or  
Lehn-Lehn-Sorger-van~Straten or 
Iliev-Ranestad or O'Grady or Iliev-Kapustka-Kapustka-Ranestad type. 
Then ${\Fn}$ is 
unirational / $\kappa({\Fn})\geq 0$ / $\kappa({\Fn})>0$ 
for the following bounds of $n$.   

\begin{center}
\begin{tabular}{cccccccccccc} 
\toprule  
     & BD   & DV   & LLSS   & IR   & OG & IKKR   
\\ \midrule  
unirational & 13 & 5 & 5 & 1 & 0 & 0  
\\ \midrule  
$\kappa \geq 0$ & 14 & 6 & 7 & 6 & 11 & 16  
\\ \midrule  
$\kappa > 0$ & 23 & 13 & 12 & 12 &  19 & 20  
\\ \bottomrule 
\end{tabular}
\end{center}

In all cases, $\kappa(\mathcal{F}_{n})=20$ when $n$ is sufficiently large. 
The stabilization $\kappa(\mathcal{F}_{n})=\dim \mathcal{M}$ at large $n$ 
holds more generally for families $\mathcal{F}\to \mathcal{M}$ 
of lattice-polarized holomorphic symplectic varieties 
whose period map is dominant and generically finite. 
\end{theorem}

This table means, for example in the Beauville-Donagi (BD) case, 
that ${\Fn}$, the moduli space of Fano varieties of cubic fourfolds with $n$ marked points 
(or equivalently cubic fourfolds with $n$ marked lines),  
is unirational when $n\leq 13$, 
has $\kappa(\mathcal{F}_{n})\geq 0$ when $n\geq 14$, 
and $\kappa(\mathcal{F}_{n}) > 0$ when $n\geq 23$. 
In particular, we find the exact transition point from $\kappa=-\infty$ to $\kappa\geq0$ 
in the Beauville-Donagi and Debarre-Voisin cases, 
and nearly exact in the Lehn-Lehn-Sorger-van Straten case. 
On the other hand, it would not be easy to explicitly calculate a bound for $\kappa=20$; 
in fact, we expect that the transition of Kodaira dimension would be sudden, 
so the actual bound for $\kappa=20$ 
would be quite near to the (actual) bound for $\kappa\geq 0$. 
(In this sense, the above bound for $\kappa>0$ should be temporary.) 
 
Markman \cite{Mark2} gave an analytic construction of 
general marked universal families over (non-Haussdorff, unpolarized) period domains. 
Here we take more ad hoc construction. 
The space ${\Fn}$ (birationally) parametrizes the isomorphism classes of 
the $n$-pointed polarized symplectic varieties except for the two double EPW cases, 
while in those cases it is a double cover of the moduli space. 

Theorem \ref{thm:main} in the direction of $\kappa\geq 0$ 
is proved by using modular forms on the period domain. 
For a family $\mathcal{F}\to \mathcal{M}$ of lattice-polarized holomorphic symplectic varieties 
of dimension $2d$ whose period map is dominant and generically finite, 
we construct an injective map (Theorem \ref{thm:cusp-canonical}) 
\begin{equation}\label{eqn:cusp canonical intro}
S_{b+dn}(\Gamma, \det) \hookrightarrow H^{0}(K_{{\Fncpt}}), 
\end{equation}
where ${\Fncpt}$ is a smooth projective model of ${\Fn}$, 
$\Gamma$ is an arithmetic group containing the monodromy group, 
$S_{k}(\Gamma, \det)$ is the space of $\Gamma$-cusp forms of weight $k$ and character $\det$, 
and $b=\dim \mathcal{M}$. 
For the above six cases, we construct cusp forms explicitly 
by using quasi-pullback of the Borcherds $\Phi_{12}$ form (\cite{Bo}, \cite{BKPSB}) 
and its product with modular forms obtained by the Gritsenko lifting (\cite{Gr}). 
The same technique of construction should be also applicable to lattice-polarized families, 
of which more examples would be available. 

The proof of unirationality is done by geometric argument, 
but in the Beauville-Donagi and Debarre-Voisin cases, 
we also make use of the ``transcendental" results 
$\kappa(\mathcal{F}^{BD}_{14})\geq 0$ and $\kappa(\mathcal{F}^{DV}_{6})\geq 0$ 
when checking nondegeneracy of certain maps in the argument (Claim \ref{lem:BD dominant}). 

A similar result has been obtained for $K3$ surfaces of low genus $g$ (\cite{Ma}), 
where the quasi-pullback $\Phi_{K3,g}$ of $\Phi_{12}$ was crucial too. 
Moreover, when $3\leq g\leq 10$, the weight of $\Phi_{K3,g}$ minus $19$ 
coincided with the dimension of a representation space appearing in the projective model of the $K3$ surfaces. 
In the present paper we see no such a direct identity, 
but a ``switched'' identity between $K3$ surfaces of genus $2$ and cubic fourfolds 
(Remark \ref{remark: g=2 K3 and cubic4}).  

This paper is organized as follows. 
\S \ref{sec:recall} is a recollection of holomorphic symplectic manifolds and modular forms. 
In \S \ref{ssec:cusp-canonical} we construct 
the map \eqref{eqn:cusp canonical intro} (Theorem \ref{thm:cusp-canonical}) 
and prove the latter half of Theorem \ref{thm:main} (Corollary \ref{cor:stabilize}). 
The first half of Theorem \ref{thm:main} is proved in 
\S \ref{sec:BD} -- \S \ref{sec:EPW}.

Throughout this paper, 
a \textit{lattice} means a free abelian group of finite rank 
endowed with a nondegenerate integral symmetric bilinear form. 
$A_{k}$, $D_{l}$, $E_{m}$ stand for the \textit{negative-definite} root lattices 
of respective types. 
The even unimodular lattice of signature $(1, 1)$ is denoted by $U$. 
No confusion will likely to occur when $U$ is also used for an open set of a variety. 
The Grassmannian parametrizing $r$-dimensional linear subspaces of 
${\C}^{N}$ is denoted by $G(r, N)=\mathbb{G}(r-1, N-1)$.  
We freely use the fact (\cite{GIT}) that 
if $G={\rm PGL}_{N}$ acts on a projective variety $X$ and 
$U$ is a $G$-invariant Zariski open set of $X$ contained in the stable locus, 
then a geometric quotient $U/G$ exists. 
If no point of $U$ has nontrivial stabilizer, $U\to U/G$ is a principal $G$-bundle in the etale topology. 
In that case, every $G$-linearized vector bundle on $U$ descends to a vector bundle on $U/G$. 
Similarly, if $V$ is a representation of ${\rm SL}_{N}$, 
a geometric quotient $({\proj}V\times U)/G$ exists 
as a Brauer-Severi variety over $U/G$.  
If $Y$ is a normal $G$-invariant subvariety of ${\proj}V\times U$, 
its geometric quotient $Y/G$ exists as the image of $Y$ in $({\proj}V\times U)/G$.

I would like to thank Kieran O'Grady for 
valuable advice on double EPW sextics. 


\section{Preliminaries}\label{sec:recall}

In this section we recall basic facts about 
holomorphic symplectic manifolds (\S \ref{ssec:HSV}) 
and orthogonal modular forms (\S \ref{ssec:modular}). 

\subsection{Holomorphic symplectic manifolds}\label{ssec:HSV}

A compact K\"ahler manifold $X$ of dimension $2d$ is called a 
\textit{holomorphic symplectic manifold} if 
it is simply connected and 
$H^{0}(\Omega_{X}^{2})={\C}\omega$ 
for a nowhere degenerate $2$-form $\omega$. 
There exists a non-divisible integral symmetric bilinear form $q_{X}$ 
of signature $(3, b_{2}(X)-3)$ on $H^{2}(X, {\Z})$, 
called the \textit{Beauville form} (\cite{Be}), 
and a constant $c_{X}$ called the \textit{Fujiki constant}, 
such that 
$\int_{X}v^{2d}=c_{X}\cdot q_{X}(v, v)^{d}$ 
for every $v\in H^{2}(X, {\Z})$. 
In particular, for $\omega\in H^{0}(\Omega_{X}^{2})$, 
we have $q_{X}(\omega, \omega)=0$ and 
\begin{equation}\label{eqn:Beauville paring 2-form}
q_{X}(\omega, \bar{\omega})^{d} = C \int_{X}(\omega \wedge \bar{\omega})^{d}
\end{equation}
for a suitable constant $C$. 

A holomorphic symplectic manifold $X$ is said to be of $K3^{[n]}$ \textit{type} 
if it is deformation equivalent to the Hilbert scheme of $n$ points on a $K3$ surface. 
The Beauville lattice of such $X$ is isometric to 
$L_{2t}=3U\oplus 2E_{8} \oplus \langle -2t \rangle$ 
where $t=n-1$ (\cite{Be}). 
Let $h\in L_{2t}$ be a primitive vector of norm $2D>0$. 
The orthogonal complement $h^{\perp}\cap L_{2t}$ is described as follows 
(\cite{GHS10} \S 3). 
For simplicity we assume $(t, D)=1$, which holds in later sections except \S \ref{ssec:IKKR}.  
We have either $(h, L_{2t})={\Z}$ or $2{\Z}$. 
In the former case, $h$ is called of \textit{split type}, and 
$h^{\perp}\cap L_{2t}$ is isometric to  
$2U \oplus 2E_{8} \oplus \langle -2t \rangle \oplus \langle -2D \rangle$. 
In the latter case, $h$ is called of \textit{non-split type}, and 
\begin{equation}\label{eqn:lattice non-split}
h^{\perp}\cap L_{2t} \simeq 
2U \oplus 2E_{8} \oplus 
\begin{pmatrix} -2t & t \\ t & -(D+t)/2 \end{pmatrix}, 
\end{equation} 
which has determinant $tD$. 
In \S \ref{sec:BD} -- \S \ref{sec:IR}, 
$h$ will be of non-split type 
and the determinant $tD$ will be a prime number of class number $1$.

\subsection{Modular forms}\label{ssec:modular}

Let $L$ be a lattice of signature $(2, b)$ with $b\geq 3$. 
The dual lattice of $L$ is denoted by $L^{\vee}$. 
We write $A_{L}=L^{\vee}/L$ for the discriminant group of $L$. 
$A_{L}$ is equipped with a natural ${\Q}/{\Z}$-valued bilinear form, 
which when $L$ is even is induced from a natural ${\Q}/2{\Z}$-valued quadratic form. 
The Hermitian symmetric domain $\mathcal{D}=\mathcal{D}_{L}$ attached to $L$ 
is defined as either of the two connected components of the space 
\begin{equation*}
\{ \: {\C}\omega \in {\proj}L_{{\C}} \: | \: 
(\omega, \omega)=0, (\omega, \bar{\omega})>0 \: \}. 
\end{equation*}
Let ${\OL}$ be the subgroup of the orthogonal group ${\rm O}(L)$ 
preserving the component $\mathcal{D}$. 
We write ${\Ost}$ for the kernel of ${\OL}\to {\rm O}(A_{L})$. 
When $A_{L} \simeq {\Z}/p$ for a prime $p$, 
which holds in \S \ref{sec:BD} -- \S \ref{sec:IR}, 
we have ${\rm O}(A_{L})=\{ \pm {\rm id} \}$ 
and so ${\OL}=\langle {\Ost}, -{\rm id} \rangle$. 

Let $\mathcal{L}$ be the restriction of the tautological line bundle 
$\mathcal{O}_{{\proj}L_{{\C}}}(-1)$ over $\mathcal{D}$. 
$\mathcal{L}$ is naturally ${\rm O}^{+}(L_{{\R}})$-linearized. 
Let $\Gamma$ be a finite-index subgroup of ${\rm O}^{+}(L)$ and  
$\chi$ be a unitary character of $\Gamma$. 
A $\Gamma$-invariant holomorphic section of 
$\mathcal{L}^{\otimes k}\otimes \chi$ 
over $\mathcal{D}$ is called a \textit{modular form} of weight $k$ 
and character $\chi$ with respect to $\Gamma$. 
When it vanishes at the cusps, it is called a \textit{cusp form} 
(see, e.g., \cite{GHS}, \cite{Ma} for the precise definition.) 
We write $M_{k}(\Gamma, \chi)$ for the space of $\Gamma$-modular forms of weight $k$ and character $\chi$, 
and $S_{k}(\Gamma, \chi)$ the subspace of cusp forms. 
We especially write $M_{k}(\Gamma)=M_{k}(\Gamma, 1)$. 
If $\Gamma' \lhd \Gamma$ is a normal subgroup of finite index, 
the quotient group $\Gamma/\Gamma'$ acts on $M_{k}(\Gamma', \chi)$ 
by translating $\Gamma'$-invariant sections by elements of $\Gamma$. 
We also remark that when $\chi=\det$ and $k \equiv b$ mod $2$, 
$-{\rm id}$ acts trivially on $\mathcal{L}^{\otimes k}\otimes \det$, 
so that 
\begin{equation}\label{eqn:-id effect}
M_{k}(\langle \Gamma, -{\rm id} \rangle, \det) = M_{k}(\Gamma, \det). 
\end{equation}
When $k \not\equiv b$ mod $2$, 
$M_{k}(\langle \Gamma, -{\rm id} \rangle, \det)$ is zero.

The Hermitian form $(\cdot, \bar{\cdot})$ on $L_{{\C}}$ 
defines an ${\rm O}^{+}(L_{{\R}})$-invariant Hermitian metric 
on the line bundle $\mathcal{L}$. 
This defines a $\Gamma$-invariant Hermitian metric on $\mathcal{L}^{\otimes k}\otimes \chi$  
which we denote by $( \: , \: )_{k,\chi}$. 
We especially write $( \: , \: )_{k}=( \: , \: )_{k,1}$. 
Let ${\rm vol}$ be the ${\rm O}^{+}(L_{{\R}})$-invariant volume form on $\mathcal{D}$, 
which exists and is unique up to constant. 

\begin{lemma}\label{lem:Petersson cusp}
Let $\mathcal{M}'$ be a Zariski open set of ${\GD}$ 
and $\mathcal{D}'\subset \mathcal{D}$ be its inverse image. 
Let $\Phi$ be a $\Gamma$-invariant holomorphic section of 
$\mathcal{L}^{\otimes k}\otimes \chi$ 
defined over $\mathcal{D}'$ with $k\geq b$. 
Then 
$\Phi\in S_{k}(\Gamma, \chi)$ 
if and only if 
$\int_{\mathcal{M}'}(\Phi, \Phi)_{k,\chi}{\rm vol} < \infty$. 
\end{lemma}

\begin{proof}
In \cite{Ma} Proposition 3.5, this is proved when $\mathcal{D}'=\mathcal{D}$, 
i.e., $\Phi\in M_{k}(\Gamma, \chi)$. 
Hence it suffices here to show that 
$\int_{\mathcal{M}'}(\Phi, \Phi)_{k,\chi}{\rm vol} < \infty$ 
implies holomorphicity of $\Phi$ over $\mathcal{D}$. 
Let $H$ be an irreducible component of $\mathcal{D}-\mathcal{D}'$. 
We may assume that $H$ is of codimension $1$. 
If $\Phi$ has a pole along $H$, say of order $a>0$, 
a local calculation shows that 
in a neighborhood of a general point of $H$, 
with $H$ locally defined by $z=0$, 
the integral 
\begin{eqnarray*}
\int_{\varepsilon \leq |z| \leq 1}  (\Phi, \Phi)_{k,\chi}{\rm vol} 
& \geq &  
C \int_{\varepsilon \leq |z| \leq 1} |z|^{-2a}dz\wedge d\bar{z} \\ 
& = &  
C \int_{0}^{2\pi} d\theta \int_{\varepsilon}^{1}r^{-2a+1}dr 
\qquad 
(z=re^{i\theta}) 
\end{eqnarray*}
must diverge as $\varepsilon \to 0$. 
\end{proof}

Let $II_{2,26}=2U\oplus 3E_{8}$ be the even unimodular lattice of signature $(2, 26)$. 
Borcherds \cite{Bo} discovered a modular form $\Phi_{12}$ of weight $12$ and character $\det$ 
for ${\rm O}^{+}(II_{2,26})$. 
The \textit{quasi-pullback} of $\Phi_{12}$ is defined as follows (\cite{Bo}, \cite{BKPSB}). 
Let $L$ be a sublattice of $II_{2,26}$ of signature $(2, b)$ and $N=L^{\perp}\cap II_{2,26}$. 
Let $r(N)$ be the number of $(-2)$-vectors in $N$. 
Then 
\begin{equation*}
\Phi_{12}|_{L} := 
\left. \frac{\Phi_{12}}{\prod_{\delta}(\delta, \cdot)} \: \right|_{{\DL}} 
\end{equation*}
where $\delta$ runs over all $(-2)$-vectors in $N$ up to $\pm 1$, 
is a nonzero modular form on ${\DL}$ of weight $12+r(N)/2$ and character $\det$ for ${\Ost}$. 
Moreover, when $r(N)>0$, $\Phi_{12}|_{L}$ is a cusp form (\cite{GHS}). 

In later sections, we will embed $h^{\perp}\cap L_{2t}$ into $II_{2,26}$ 
by embedding the last rank $2$ component of \eqref{eqn:lattice non-split} into $E_{8}$. 
The following model of $E_{8}$ will be used: 
\begin{equation}\label{eqn:E8}
E_{8} = 
\{ \: (x_{i})\in {\Q}^{8} \: | \: 
\forall x_{i}\in {\Z} \: \textrm{or} \: \forall x_{i}\in {\Z}+1/2, \; 
x_{1}+ \cdots +x_{8}\in 2{\Z} \: \}.  
\end{equation}
Here we take the standard (negative) quadratic form on ${\Q}^{8}$. 
The $(-2)$-vectors in $E_{8}$ are as follows. 
For $j\ne k$ we define  
$\delta_{\pm j, \pm k}=(x_{i})$ by 
$x_{j}=\pm 1$, $x_{k}=\pm 1$ and $x_{i}=0$ for $i\ne j, k$. 
For a subset $S$ of $\{ 1, \cdots, 8 \}$ consisting of even elements, 
we define 
$\delta'_{S}=(x_{i})$ by  
$x_{i}=1/2$ if $i\in S$ and $x_{i}=-1/2$ if $i\not\in S$. 
These are the $240$ roots of $E_{8}$. 
 
We will also use the Gritsenko lifting \cite{Gr}. 
Assume that $L$ is even and contains $2U$. 
We shall specialize to the case $b=20$ for later use. 
For an odd number $k$, 
let $M_{k}(\rho_{L})$ be the space of modular forms for ${\rm SL}_{2}({\Z})$ 
of weight $k$ with values in the Weil representation $\rho_{L}$ on ${\C}A_{L}$. 
The Gritsenko lifting with $b=20$ is an injective, ${\rm O}^{+}(L)$-equivariant map 
\begin{equation*}
 M_{k}(\rho_{L}) \hookrightarrow M_{k+9}({\Ost}). 
\end{equation*}
The dimension of $M_{k}(\rho_{L})$ for $k>2$ can be explicitly computed 
by using the formula in \cite{Br}. 
A similar formula for the ${\rm O}(A_{L})$-invariant part 
$M_{k}(\rho_{L})^{{\rm O}(A_{L})}$ 
is given in \cite{Ma2}.


\section{Cusp forms and canonical forms}\label{ssec:cusp-canonical} 

In this section we establish, in a general setting, 
a correspondence between 
canonical forms on $n$-fold universal family of holomorphic symplectic varieties 
and modular forms on the period domain. 
This is the basis of this paper. 
As a consequence we deduce in Corollary \ref{cor:stabilize} the latter half of Theorem \ref{thm:main}. 
The first half of Theorem \ref{thm:main} will be proved case-by-case in later sections. 

Let $M$ be a hyperbolic lattice and 
$L$ be a lattice of signature $(2, b)$. 
We say that a smooth algebraic family 
$\pi:\mathcal{F}\to \mathcal{M}$ 
of holomorphic symplectic manifolds is 
\textit{$M$-polarized with polarized Beauville lattice $L$} 
if $R^{2}\pi_{\ast}{\Z}$ 
contains a sub local system $\Lambda_{pol}$ in its $(1, 1)$-part 
whose fiber is isometric to $M$ 
with the orthogonal complement isometric to $L$. 
Let $\Lambda_{per}=(\Lambda_{pol})^{\perp}\cap R^{2}\pi_{\ast}{\Z}$ 
and we choose an isometry $(\Lambda_{per})_{x_{0}}\simeq L$ at some base point $x_{0}\in \mathcal{M}$. 
If a finite-index subgroup $\Gamma$ of ${\OL}$ 
contains the monodromy group of $\Lambda_{per}$, 
we can define the period map 
\begin{equation*}
\mathcal{P} : 
\mathcal{M} \to \Gamma \backslash {\DL}, \qquad 
x \mapsto [H^{2,0}(\mathcal{F}_{x}) \subset (\Lambda_{per})_{x}\otimes {\C}]. 
\end{equation*}
By Borel's extension theorem, $\mathcal{P}$ is a morphism of algebraic varieties. 
Our interest will be in the case ${\rm rk}(M)=1$, 
but the proof of the following theorem works in the general lattice-polarized setting as well. 

\begin{theorem}\label{thm:cusp-canonical}
Let $L$ be a lattice of signature $(2, b)$ 
and $\Gamma$ be a finite-index subgroup of ${\rm O}^{+}(L)$. 
Let $\mathcal{F}\to \mathcal{M}$ 
be a smooth algebraic family of 
lattice-polarized holomorphic symplectic manifolds of dimension $2d$ 
with polarized Beauville lattice $L$  
whose monodromy group is contained in $\Gamma$. 
Assume that the period map 
$\mathcal{P}:\mathcal{M}\to {\GD}$ 
is dominant and generically finite. 
If  
${\Fn}=\mathcal{F}\times_{\mathcal{M}}\cdots \times_{\mathcal{M}}\mathcal{F}$ 
($n$ times) and 
${\Fncpt}$ is a smooth projective model of ${\Fn}$,  
we have a natural injective map  
\begin{equation}\label{eqn:general correspondence}
S_{b+dn}(\Gamma, \det) \hookrightarrow H^{0}(K_{{\Fncpt}}) 
\end{equation}
which makes the following diagram commutative:  
\begin{equation}\label{eqn:CD canonical map}
\xymatrix{
{\Fncpt} \ar@{-->}[r]^{\phi_{K}} \ar@{-->}[d] & 
|K_{{\Fncpt}}|^{\vee}  \ar@{-->}[d]^{\eqref{eqn:general correspondence}^{\vee}} \\ 
{\GD} \ar@{-->}[r]_{\phi} & {\proj}S_{b+dn}(\Gamma, \det)^{\vee}.  
}
\end{equation}
Here 
$\phi_{K}$ is the canonical map of ${\Fncpt}$ and 
$\phi$ is the rational map defined by the sections in $S_{b+dn}(\Gamma, \det)$. 
Furthermore, if the period map $\mathcal{P}$ is birational and $\Gamma$ does not contain $-{\rm id}$, 
\eqref{eqn:general correspondence} is an isomorphism. 
\end{theorem}

\begin{proof}
Let $\mathcal{M}'=\mathcal{P}(\mathcal{M})\subset{\GD}$ 
and $\mathcal{D}'\subset \mathcal{D}$ be the inverse image of $\mathcal{M}'$.   
Shrinking $\mathcal{M}$ as necessary, 
we may assume that 
both $\mathcal{M}\to \mathcal{M}'$ and $\mathcal{D}'\to \mathcal{M}'$ are unramified. 
We take the universal cover 
$\tilde{\mathcal{M}}\to \mathcal{M}$ of $\mathcal{M}$ 
and pullback the family:  
write 
$\tilde{\mathcal{F}}=\mathcal{F}\times_{\mathcal{M}}\tilde{\mathcal{M}}$  
with the projection 
$\pi\colon \tilde{\mathcal{F}} \to \tilde{\mathcal{M}}$. 
We obtain a lift 
$\tilde{\mathcal{P}} : \tilde{\mathcal{M}} \to \mathcal{D}'\subset \mathcal{D}$ 
of the period map $\mathcal{P}$ 
which is equivariant with respect to the monodromy representation 
$\pi_{1}(\mathcal{M})\to \Gamma$. 
Since $\mathcal{P}$ is unramified, 
$\tilde{\mathcal{P}}$ is unramified too.  
We first construct an injective map 
\begin{equation}\label{eqn:correspondence interior}
H^{0}(\mathcal{D}', \mathcal{L}^{\otimes b+dn}\otimes \det)^{\Gamma} 
\hookrightarrow  
H^{0}({\Fn}, K_{{\Fn}}),  
\end{equation}
where $H^{0}({\Fn}, K_{{\Fn}})$ means the space of 
holomorphic (rather than regular) canonical forms on ${\Fn}$. 

We have a natural ${\rm O}^{+}(L_{{\R}})$-equivariant isomorphism 
$K_{\mathcal{D}} \simeq \mathcal{L}^{\otimes b}\otimes \det$ 
of line bundles over $\mathcal{D}$ (see, e.g., \cite{GHS}, \cite{Ma}), 
and hence a $\pi_{1}(\mathcal{M})$-equivariant isomorphism 
\begin{equation}\label{eqn:KC}
K_{\tilde{\mathcal{M}}} 
\simeq \tilde{\mathcal{P}}^{\ast}K_{\mathcal{D}} 
\simeq \tilde{\mathcal{P}}^{\ast}(\mathcal{L}^{\otimes b}\otimes \det)  
\end{equation}
over $\tilde{\mathcal{M}}$. 
Here $\pi_{1}(\mathcal{M})$ acts on 
$\tilde{\mathcal{P}}^{\ast}\mathcal{L}$, $\tilde{\mathcal{P}}^{\ast}\det$ 
through the $\Gamma$-action on $\mathcal{L}$, $\det$ 
and the monodromy representation $\pi_{1}(\mathcal{M})\to \Gamma$. 

On the other hand, 
by the definition of the period map, 
we have a canonical isomorphism 
$\pi_{\ast}\Omega_{\pi}^{2} \simeq \tilde{\mathcal{P}}^{\ast}\mathcal{L}$ 
sending a symplectic form to its cohomology class. 
Since $\pi\colon \tilde{\mathcal{F}}\to \tilde{\mathcal{M}}$ 
is a family of holomorphic symplectic manifolds, 
both $\pi_{\ast}\Omega_{\pi}^{2}$ and 
$\pi_{\ast}K_{\pi}$ are invertible sheaves, 
and the homomorphism 
$(\pi_{\ast}\Omega_{\pi}^{2})^{\otimes d} \to \pi_{\ast}K_{\pi}$ 
defined by the wedge product is isomorphic. 
Therefore we have a natural isomorphism  
$\pi_{\ast}K_{\pi} \simeq \tilde{\mathcal{P}}^{\ast}\mathcal{L}^{\otimes d}$. 
Since the natural homomorphism 
$\pi^{\ast}\pi_{\ast}K_{\pi}\to K_{\pi}$ is isomorphic, 
we find that 
$K_{\pi} \simeq \pi^{\ast}\tilde{\mathcal{P}}^{\ast}\mathcal{L}^{\otimes d}$. 
By construction this is $\pi_{1}(\mathcal{M})$-equivariant. 
If we write 
$\tilde{\mathcal{F}}_{n}=\mathcal{F}_{n}\times_{\mathcal{M}}\tilde{\mathcal{M}}$ 
with the projection 
$\pi_{n}\colon \tilde{\mathcal{F}}_{n} \to \tilde{\mathcal{M}}$, 
this shows that 
\begin{equation}\label{eqn:isom Kpin univ cover}
K_{\pi_{n}} \simeq \pi_{n}^{\ast}\tilde{\mathcal{P}}^{\ast} \mathcal{L}^{\otimes dn} 
\end{equation} 
as $\pi_{1}(\mathcal{M})$-linearized line bundles on $\tilde{\mathcal{F}}_{n}$. 
Combining \eqref{eqn:KC} and \eqref{eqn:isom Kpin univ cover}, 
we obtain a $\pi_{1}(\mathcal{M})$-equivariant isomorphism 
\begin{equation*}
K_{\tilde{\mathcal{F}}_{n}} \simeq 
\pi_{n}^{\ast}\tilde{\mathcal{P}}^{\ast}(\mathcal{L}^{\otimes b+dn}\otimes \det) 
\end{equation*}
over $\tilde{\mathcal{F}}_{n}$. 
Hence pullback of 
sections of $\mathcal{L}^{\otimes b+dn}\otimes \det$ over $\mathcal{D}'$ 
by $\tilde{\mathcal{P}}\circ \pi_{n}$ defines 
a $\pi_{1}(\mathcal{M})$-equivariant injective map 
\begin{equation}\label{eqn:isom univ cover}
H^{0}(\mathcal{D}', \mathcal{L}^{\otimes b+dn}\otimes \det) 
\hookrightarrow 
H^{0}(\tilde{\mathcal{F}}_{n}, K_{\tilde{\mathcal{F}}_{n}}).  
\end{equation}
Taking the invariant parts by $\Gamma$ and $\pi_{1}(\mathcal{M})$ respectively, 
we obtain \eqref{eqn:correspondence interior}. 


Next we prove that restriction of \eqref{eqn:correspondence interior} 
gives the desired map \eqref{eqn:general correspondence}. 
Let $\Phi$ be a $\Gamma$-invariant section of 
$\mathcal{L}^{\otimes b+dn}\otimes \det$ over $\mathcal{D}'$ and 
$\omega \in H^{0}(K_{{\Fn}})$ be the image of $\Phi$ by \eqref{eqn:correspondence interior}. 
We shall show that 
\begin{equation*}
\int_{{\Fn}}\omega\wedge\bar{\omega} = 
C \int_{\mathcal{M}'}(\Phi, \Phi)_{b+dn, \det} {\rm vol}
\end{equation*}
for some constant $C$. 
Our assertion then follows from Lemma \ref{lem:Petersson cusp} 
and the standard fact that 
$\omega$ extends over a smooth projective model of ${\Fn}$ 
if and only if  
$\int_{{\Fn}}\omega\wedge\bar{\omega}<\infty$. 

Since the problem is local, 
it suffices to take an arbitrary small open set 
$U\subset \tilde{\mathcal{M}}$ and prove 
\begin{equation}\label{eqn:Petersson=L2}
\int_{\pi_{n}^{-1}(U)}\omega\wedge\bar{\omega} = 
C \int_{\tilde{\mathcal{P}}(U)}(\Phi, \Phi)_{b+dn, \det}{\rm vol} 
\end{equation}
for some constant $C$ independent of $U$. 
In what follows, $C$ stands for any unspecified such a constant. 
Since $U$ is small, we may decompose $\Phi$ as 
$\Phi=\Phi_{1}\otimes \Phi_{2}^{\otimes dn}$ 
with $\Phi_{1}$ a local section of $\mathcal{L}^{\otimes b}\otimes \det$ 
and $\Phi_{2}$ a local section of $\mathcal{L}$. 
Let $\omega_{1}$ be the canonical form on 
$U\simeq \tilde{\mathcal{P}}(U)$ corresponding to $\Phi_{1}$, and 
$\omega_{2}$ be the relative symplectic form on 
$\tilde{\mathcal{F}}|_{U} \to U$ corresponding to $\tilde{\mathcal{P}}^{\ast}\Phi_{2}$. 
On the one hand, we have 
\begin{equation}\label{eqn:Petersson=L2 base}
\omega_{1}\wedge \bar{\omega}_{1} = 
C (\Phi_{1}, \Phi_{1})_{b,\det}{\rm vol} 
\end{equation}
(see, e.g., \cite{Ma} \S 3.1). 
On the other hand, at each fiber $X$ of $\tilde{\mathcal{F}}|_{U}$, 
the pointwise Petersson norm $(\Phi_{2}, \Phi_{2})_{1}=(\Phi_{2}, \bar{\Phi}_{2})$ 
is nothing but the pairing $q_{X}(\omega_{2}, \bar{\omega}_{2})$ in the Beauville form of $X$. 
Since 
\begin{equation*}
q_{X}(\omega_{2}, \bar{\omega}_{2})^{d} = 
C \int_{X}(\omega_{2} \wedge \bar{\omega}_{2})^{d}  
\end{equation*}
by \eqref{eqn:Beauville paring 2-form}, 
we find that 
\begin{equation}\label{eqn:Petersson=L2 fiber}
(\Phi_{2}^{\otimes dn}, \Phi_{2}^{\otimes dn})_{dn} = 
C \int_{X^{n}} (p_{1}^{\ast}\omega_{2}\wedge \cdots \wedge p_{n}^{\ast}\omega_{2})^{d} \wedge 
(p_{1}^{\ast}\bar{\omega}_{2} \wedge  \cdots \wedge p_{n}^{\ast}\bar{\omega}_{2})^{d}, 
\end{equation}
where $p_{i}\colon X^{n}\to X$ is the $i$-th projection. 
Since $(p_{1}^{\ast}\omega_{2}\wedge \cdots \wedge p_{n}^{\ast}\omega_{2})^{d}$ 
is the canonical form on $X^{n}$ corresponding to
the value of $\Phi_{2}^{\otimes dn}$ at $[X]\in U$, 
the equalities \eqref{eqn:Petersson=L2 base} and \eqref{eqn:Petersson=L2 fiber} 
imply \eqref{eqn:Petersson=L2}. 
Thus we obtain the map \eqref{eqn:general correspondence}. 
Since this map is defined by pullback of sections of line bundle, 
the diagram \eqref{eqn:CD canonical map} is commutative. 

Finally, when $\mathcal{P}$ is birational, 
we may assume as before that it is an open immersion. 
If $\Gamma$ does not contain $-{\rm id}$, 
$\Gamma$ acts on $\mathcal{D}$ effectively, 
and the monodromy group coincides with $\Gamma$. 
We can kill the monodromy by pulling back the family $\mathcal{F}\to \mathcal{M}$ 
to $\mathcal{D}'$ instead of to $\tilde{\mathcal{M}}$. 
Rewriting $\tilde{\mathcal{F}}_{n}={\Fn}\times_{\mathcal{M}}\mathcal{D}'$, 
this shows that \eqref{eqn:isom univ cover} is isomorphic. 
Taking the $\Gamma$-invariant part, 
we see that \eqref{eqn:correspondence interior} is isomorphic. 
Finally, taking the subspace of finite norm, 
we see that \eqref{eqn:general correspondence} is isomorphic. 
This completes the proof of Theorem \ref{thm:cusp-canonical}.  
\end{proof}

\begin{remark}
The last statement of Theorem \ref{thm:cusp-canonical} can also be proved more directly 
by using descends of the $\Gamma$-linearized line bundles $\mathcal{L}$, $\det$ 
to line bundles on $\mathcal{M}\subset {\GD}$. 
\end{remark}

\begin{corollary}\label{cor:stabilize} 
If $n$ is sufficiently large, then $\kappa({\Fn})=b$. 
\end{corollary}

\begin{proof}
Since ${\Fn}\to\mathcal{F}_{n-1}$ is 
a smooth family of holomorphic symplectic varieties, 
$\kappa({\Fn})$ is nondecreasing with respect to $n$ 
by Iitaka's subadditivity conjecture known in this case \cite{Ka}. 
We also have the bound 
$\kappa({\Fn})\leq \dim \mathcal{M} = b$ 
by Iitaka's addition formula \cite{Ii}. 

We take a weight $k_{0}$ such that 
$S_{k_{0}}(\Gamma, \det)\ne \{ 0 \}$. 
Then we take a weight $k_{1}$ such that $k_{1}\equiv b-k_{0}$ mod $d$ and that 
${\GD}\dashrightarrow {\proj}M_{k_{1}}(\Gamma)^{\vee}$ 
is generically finite onto its image. 
(When $\Gamma$ contains $-{\rm id}$, we must have $k_{0}\equiv b$ mod $2$, 
so $b-k_{0}+d{\Z}$ contains sufficiently large even $k_{1}$.) 
Since 
\begin{equation*}
S_{k_{0}}(\Gamma, \det)\cdot M_{k_{1}}(\Gamma) \subset S_{k_{0}+k_{1}}(\Gamma, \det), 
\end{equation*}
Theorem \ref{thm:cusp-canonical} implies that 
for $n_{0}=(k_{0}+k_{1}-b)/d$, 
the image of the canonical map of $\bar{\mathcal{F}}_{n_{0}}$ has dimension $\geq b$. 
Hence $\kappa(\mathcal{F}_{n_{0}})\geq b$ 
and so   
$\kappa({\Fn})=b$ for all $n\geq n_{0}$. 
\end{proof}

This proves the latter half of Theorem \ref{thm:main}. 
In the following sections, we apply Theorem \ref{thm:cusp-canonical} 
to the six explicit families in Theorem \ref{thm:main}. 
In practice, one needs to identify the group $\Gamma$. 
For example, according to \cite{Mark} Remark 8.5 and \cite{GHS} Remark 3.15, 
the monodromy group of a family of 
polarized symplectic manifolds of $K3^{[2]}$ type with polarization vector $h$ 
is contained in $\tilde{{\rm O}}^{+}(h^{\perp}\cap L_{2})$.


\section{Fano varieties of cubic fourfolds}\label{sec:BD}

In this section we prove Theorem \ref{thm:main} 
for the case of Fano varieties of cubic fourfolds \cite{BD}. 
Let $Y\subset {\proj}^{5}$ be a smooth cubic fourfold. 
The Fano variety $F(Y)\subset {\G}$ of $Y$ is the variety parametrizing lines on $Y$, 
which is smooth of dimension $4$. 
Beauville-Donagi \cite{BD} proved that 
$F(Y)$ is a holomorphic symplectic manifold of $K3^{[2]}$ type 
polarized by the Pl\"ucker,  
and its polarized Beauville lattice is isometric to $L_{cub}=2U\oplus 2E_{8}\oplus A_{2}$. 
In fact, the polarized Beauville lattice of $F(Y)$ 
is isomorphic to the primitive part of $H^{4}(Y, {\Z})$ as polarized Hodge structures, 
where the intersection form on $H^{4}(Y, {\Z})$ is $(-1)$-scaled. 

Let $U\subset |\mathcal{O}_{{\proj}^{5}}(3)|$ 
be the parameter space of smooth cubic fourfolds. 
By GIT (\cite{GIT}), the geometric quotient $U/{\rm PGL}_{6}$ 
exists as an affine variety of dimension $20$. 
Let $\Gamma=\tilde{{\rm O}}^{+}(L_{cub})$. 
The period map 
$U/{\rm PGL}_{6} \to {\GD}$ 
is an open immersion by Voisin \cite{Vo}, 
and the complement of its image was determined by 
Looijenga \cite{Lo} and Laza \cite{La}.  

\begin{lemma}[cf.~\cite{La}]\label{lem:cusp form BD}
The cusp form $\Phi_{12}|_{L_{cub}}$ has weight $48$. 
Moreover, $S_{66}(\Gamma, \det)$ and $S_{68}(\Gamma, \det)$ 
have dimension $\geq 2$. 
\end{lemma}

\begin{proof}
Write $L=L_{cub}$. 
The weight of $\Phi_{12}|_{L}$ is computed in \cite{La}. 
($A_{2}^{\perp}\simeq E_{6}$ has $72$ roots.) 
We have 
$\dim M_{k}(\rho_{L})=[(k+3)/6]$ 
by computing the formula in \cite{Br}. 
Product of $\Phi_{12}|_{L}$ with the Gritsenko lift of 
$M_{9}(\rho_{L})$ and $M_{11}(\rho_{L})$ 
proves the second assertion. 
\end{proof}

We consider the parameter space of 
smooth cubic fourfolds with $n$ marked lines: 
\begin{equation*}
F_{n} 
 =  
\{ \: (Y, l_{1}, \cdots , l_{n}) \: | \: 
Y\in U, \: l_{1}, \cdots, l_{n} \in F(Y) \: \} \\ 
 \subset  
U\times {\G}^{n}, 
\end{equation*}
and let   
$\mathcal{F}_{n} = F_{n}/{\rm PGL}_{6}$.  
Then $\mathcal{F}_{n}$ is smooth over the open locus of $U/{\rm PGL}_{6}$ 
where cubic fourfolds have no nontrivial stabilizer. 
By Lemma \ref{lem:cusp form BD}, 
with $48=20+2\cdot 14$ and $66=20+2\cdot 23$,  
we see that $\mathcal{F}_{14}$ has positive geometric genus 
and $\kappa(\mathcal{F}_{23})>0$. 
(Cusp forms of weight $68$ will be used in \S \ref{sec:LLSS}.) 
%
%
It remains to prove that $\mathcal{F}_{13}$ is unirational. 
We prove  

\begin{proposition}\label{prop:F5 BD rational}
$F_{13}$ is rational. 
\end{proposition}

\begin{proof}
Consider the second projection $\pi\colon F_{13}\to {\G}^{13}$. 
If $(l_{1}, \cdots, l_{13})\in \pi(F_{13})$, 
the fiber $\pi^{-1}(l_{1}, \cdots, l_{13})$ is a non-empty open set of 
the linear system of cubics containing $l_{1}, \cdots, l_{13}$, 
which we denote by  
\begin{equation*}
{\proj}V(l_{1},\cdots,l_{13}) = 
{\proj}\ker (H^{0}(\mathcal{O}_{{\proj}^{5}}(3)) \to 
\oplus_{i=1}^{13}H^{0}(\mathcal{O}_{l_{i}}(3))). 
\end{equation*}
This shows that $F_{13}$ is birationally a ${\proj}^{N}$-bundle over $\pi(F_{13})$ with 
\begin{equation*}
N= \dim F_{13} - \dim \pi(F_{13}) \geq \dim F_{13} - \dim {\G}^{13} = 3. 
\end{equation*}
Hence we are reduced to the following assertion.

\begin{claim}\label{lem:BD dominant}
$\pi\colon F_{13} \to {\G}^{13}$ is dominant. 
\end{claim}

%
Assume to the contrary that $\pi$ was not dominant. 
Then we have 
$\dim V(l_{1},\cdots,l_{13})\geq 5$ 
for a general point $(l_{1}, \cdots, l_{13})$ of $\pi(F_{13})$. 
Consider the similar projection 
$\pi'\colon F_{14}\to {\G}^{14}$ in $n=14$. 
Since $\mathcal{F}_{14}=F_{14}/{\rm PGL}_{6}$ cannot be uniruled as just proved, 
we must have  
$\dim V(l_{1}, \cdots, l_{14})=1$ 
for general $(l_{1}, \cdots, l_{14})\in \pi'(F_{14})$. 
On the other hand, $V(l_{1}, \cdots, l_{14})$ can be written as 
\begin{equation*}
V(l_{1},\cdots,l_{14}) = 
\ker (V(l_{1},\cdots,l_{13}) \stackrel{\rho}{\to} H^{0}(\mathcal{O}_{l_{14}}(3))), 
\end{equation*}
where $\rho$ is the restriction map. 
Hence 
for general $(l_{1}, \cdots, l_{13})\in \pi(F_{13})$, 
we have 
$\dim V(l_{1}, \cdots, l_{13})=5$, 
$\rho$ is surjective, 
and $\pi(F_{13})$ is of codimension $1$ in ${\G}^{13}$. 

The last property implies that the similar projection 
$\pi''\colon F_{12}\to {\G}^{12}$ in $n=12$ must be dominant, 
because otherwise $\pi(F_{13})$ would be dense in the inverse image of 
$\pi''(F_{12})\subset {\G}^{12}$ 
by the projection ${\G}^{13}\to {\G}^{12}$, 
which contradicts the $\frak{S}_{13}$-invariance of $\pi(F_{13})$. 
This in turn shows that 
\begin{equation*}
\dim V(l_{1},\cdots,l_{12}) = \dim F_{12} - \dim {\G}^{12} + 1 = 8 
\end{equation*}
for a general point $(l_{1},\cdots,l_{12})$ of ${\G}^{12}$. 
However, since 
$V(l_{1},\cdots,l_{13})\to H^{0}(\mathcal{O}_{l_{14}}(3))$ 
is surjective, 
$V(l_{1},\cdots,l_{12})\to H^{0}(\mathcal{O}_{l_{14}}(3))$ 
is surjective too. 
Hence 
$\dim V(l_{1},\cdots,l_{12},l_{14})=4$. 
But since $(l_{1}, \cdots, l_{12}, l_{14})$ is a general point of $\pi(F_{13})$, 
this is absurd. 
This proves Claim \ref{lem:BD dominant} 
and so finishes the proof of Proposition \ref{prop:F5 BD rational}. 
\end{proof}

\begin{remark}\label{remark: g=2 K3 and cubic4}
In the analogous case of $K3$ surfaces of genus $g$ (\cite{Ma}), 
when $3\leq g \leq 10$, 
the weight of the quasi-pullback $\Phi_{K3,g}$ of $\Phi_{12}$ coincided with 
\begin{equation*}
{\rm weight}(\Phi_{K3,g}) = \dim V_{g} + 19 = \dim V_{g} + \dim({\rm moduli}) 
\end{equation*}
for a representation space $V_{g}$ related to the projective model of the $K3$ surfaces. 
Here, 
for $\Phi_{K3,2}$ and $\Phi_{cubic}=\Phi_{12}|_{L_{cub}}$, 
the ``switched'' equalities 
\begin{eqnarray*}
& & {\rm weight}(\Phi_{K3,2}) -19   =  56  =  h^{0}(\mathcal{O}_{{\proj}^{5}}(3))   \\
& & {\rm weight}(\Phi_{cubic}) - 20 =  28  =  h^{0}(\mathcal{O}_{{\proj}^{2}}(6)) 
\end{eqnarray*}
hold. 
Is this accidental? 
\end{remark}


\section{Debarre-Voisin fourfolds}\label{sec:DV}

In this section we prove Theorem \ref{thm:main} 
for the case of Debarre-Voisin fourfolds \cite{DV}. 
Let $\mathcal{E}$ be the dual of the rank $6$ universal sub vector bundle 
over the Grassmannian $G(6, 10)$. 
The space $H^{0}(\bigwedge^{3}\mathcal{E})$ is naturally isomorphic to $\bigwedge^{3}({\C}^{10})^{\vee}$. 
Debarre-Voisin \cite{DV} proved that the zero locus 
$X_{\sigma}\subset G(6, 10)$ of a general section $\sigma$ of $\bigwedge^{3}\mathcal{E}$ 
is a holomorphic symplectic manifold of $K3^{[2]}$ type, 
and the polarization given by the Pl\"ucker has Beauville norm $22$ and is of non-split type. 
The polarized Beauville lattice is hence isometric to 
\begin{equation*}\label{eqn:LDV} 
L_{DV} = 2U \oplus 2E_{8} \oplus K, \qquad 
K=\begin{pmatrix} -2 & 1 \\ 1 & -6 \end{pmatrix}. 
\end{equation*}
Let $\Gamma=\tilde{{\rm O}}^{+}(L_{DV})$. 

\begin{lemma}\label{lem:cusp form DV}
There exists an embedding $K\hookrightarrow E_{8}$ with $r(K^{\perp})=40$. 
The resulting cusp form $\Phi_{12}|_{L_{DV}}$ has weight $32$. 
Moreover, $S_{46}(\Gamma, \det)$ has dimension $\geq 2$. 
\end{lemma}

\begin{proof}
Let $v_{1}, v_{2}$ be the basis of $K$ in the above matrix expression. 
We embed $K$ into $E_{8}$, in the model \eqref{eqn:E8} of $E_{8}$, by 
\begin{equation*}
v_{1}\mapsto (1, -1, 0, \cdots, 0), \quad 
v_{2}\mapsto (0, 1, 1, 2, 0, \cdots, 0). 
\end{equation*}
The roots of $E_{8}$ orthogonal to these two vectors are 
$\delta_{\pm i, \pm j}$ with $i, j\geq 5$ and 
$\pm\delta'_{S}$ with $1, 2, 3\in S$ and $4\not\in S$. 
The total number is $24+16=40$. 
Hence the weight of $\Phi_{12}|_{L_{DV}}$ is $12+20=32$. 
Working out the formula in \cite{Br}, 
we also see that $\dim M_{k}(\rho_{L_{DV}})=(k-1)/2$. 
Taking product of $\Phi_{12}|_{L_{DV}}$ with 
the Gritsenko lift of $M_{5}(\rho_{L_{DV}})$, 
we obtain the last assertion. 
\end{proof}
 
Let $U$ be the open locus of ${\proj}(\bigwedge^{3}{\C}^{10})^{\vee}$ 
where $X_{\sigma}$ is smooth of dimension $4$  
and $[\sigma]$ is ${\rm PGL}_{10}$-stable with no nontrivial stabilizer.  
The period map 
$U/{\rm PGL}_{10}\to {\GD}$ 
is generically finite and dominant (\cite{DV}). 
Consider the incidence 
\begin{equation*}
F_{n} = 
\{ \: ([\sigma], p_{1}, \cdots, p_{n}) \in U \times G(6, 10)^{n} \: | \: p_{i}\in X_{\sigma} \: \} 
\subset U \times G(6, 10)^{n} 
\end{equation*}
and let 
$\mathcal{F}_{n}=F_{n}/{\rm PGL}_{10}$. 
By Lemma \ref{lem:cusp form DV}, 
with $32=20+2\cdot 6$ and $46=20+2\cdot 13$, 
we see that $\mathcal{F}_{6}$ has positive geometric genus and $\kappa(\mathcal{F}_{13})>0$. 
It remains to show that $\mathcal{F}_{5}$ is unirational. 
We prove 

\begin{proposition}
$F_{5}$ is rational. 
\end{proposition}

\begin{proof}
Consider the second projection 
$\pi\colon F_{n}\to G(6, 10)^{n}$. 
The fiber $\pi^{-1}(p_{1}, \cdots, p_{n})$ over 
$(p_{1}, \cdots, p_{n})\in \pi(F_{n})$ 
is a non-empty open set of 
the linear system 
${\proj}V(p_{1},\cdots,p_{n})\subset {\proj}H^{0}(\bigwedge^{3}\mathcal{E})$ 
of sections vanishing at $p_{1}, \cdots, p_{n}$. 
When $n=5$, we have 
\begin{equation*}
\dim V(p_{1},\cdots,p_{5}) \geq 
h^{0}(\wedge^{3}\mathcal{E}) - 5 \cdot {\rm rk}(\wedge^{3}\mathcal{E}) = 20, 
\end{equation*}
so 
$F_{5}\to \pi(F_{5})$ is birationally a ${\proj}^{N}$-bundle with $N\geq 19$. 
Furthermore, by the same argument as Claim \ref{lem:BD dominant}, 
the above result $\kappa(\mathcal{F}_{6})\geq 0$ enables us to conclude 
that $F_{5}\to G(6, 10)^{5}$ is dominant. 
Therefore $F_{5}$ is rational. 
\end{proof}



\section{Lehn-Lehn-Sorger-van Straten eightfolds}\label{sec:LLSS}

In this section we prove Theorem \ref{thm:main} 
for the case of Lehn-Lehn-Sorger-van~Straten eightfolds \cite{LLSS}. 
They have the same parameter space and period space as the Beauville-Donagi case. 

Let $Y\subset {\proj}^{5}$ be a smooth cubic fourfold which does not contain a plane. 
The space $M^{gtc}(Y)$ of \textit{generalized twisted cubics} on $Y$ is defined as 
the closure of the locus of twisted cubics on $Y$ in the Hilbert scheme ${\rm Hilb}_{3m+1}(Y)$. 
By Lehn-Lehn-Sorger-van Straten \cite{LLSS}, 
$M^{gtc}(Y)$ is smooth and irreducible of dimension $10$, and 
there exists a natural contraction $M^{gtc}(Y)\to X(Y)$ 
to a holomorphic symplectic manifold $X(Y)$ 
with general fibers ${\proj}^{2}$. 
The variety $X(Y)$ is of $K3^{[4]}$ type (\cite{AL}) 
and has a polarization of Beauville norm $2$ and non-split type 
(see \cite{De} footnote 22). 
Hence its polarized Beauville lattice is isometric to the lattice 
$L_{cub}=2U \oplus 2E_{8}\oplus A_{2}$ 
considered in \S \ref{sec:BD}, 
and the monodromy group is evidently contained in ${\rm O}^{+}(L_{cub})$. 
We can reuse Lemma \ref{lem:cusp form BD}: 
since 
${\rm O}^{+}(L_{cub})= \langle \tilde{{\rm O}}^{+}(L_{cub}), -{\rm id} \rangle$ 
and the weights in Lemma \ref{lem:cusp form BD} are even, 
the cusp forms there are 
not just $\tilde{{\rm O}}^{+}(L_{cub})$-invariant 
but also ${\rm O}^{+}(L_{cub})$-invariant 
as remarked in \eqref{eqn:-id effect}.  

Let $H={\rm Hilb}^{gtc}({\proj}^{5})$ be the irreducible component of 
the Hilbert scheme ${\rm Hilb}_{3m+1}({\proj}^{5})$ 
that contains the locus of twisted cubics in ${\proj}^{5}$. 
Then $H$ is smooth of dimension $20$, 
and we have $M^{gtc}(Y)=H\cap {\rm Hilb}_{3m+1}(Y)$ 
for $Y$ as above (\cite{LLSS}). 
Let $U\subset |\mathcal{O}_{{\proj}^{5}}(3)|$ 
be the parameter space of smooth cubic fourfolds 
which does not contain a plane and has no nontrivial stabilizer in ${\rm PGL}_{6}$. 
The period map $U/{\rm PGL}_{6} \to {\GD}$, 
where $\Gamma={\rm O}^{+}(L_{cub})$, 
is generically finite and dominant (\cite{LLSS}, \cite{AL}). 
We consider the incidence 
\begin{equation*}
M^{gtc}_{n} = 
\{ \: (Y, C_{1}, \cdots, C_{n})\in U\times H^{n} \: | \: C_{i}\in M^{gtc}(Y) \: \} 
\subset U\times H^{n}. 
\end{equation*}
As noticed in \cite{LLSS}, 
the construction of $X(Y)$ can be done in family. 
This produces a smooth family $X\to U$ of symplectic eightfolds 
and a contraction $M^{gtc}_{1}\to X$ over $U$ 
with general fibers ${\proj}^{2}$. 
Taking the $n$-fold fiber product 
$X_{n} = X\times_{U} \cdots \times_{U} X$, 
we obtain a morphism 
$M^{gtc}_{n}\to X_{n}$ over $U$ 
with general fibers $({\proj}^{2})^{n}$. 
Let 
$\mathcal{F}_{n}=X_{n}/{\rm PGL}_{6}$. 
By Lemma \ref{lem:cusp form BD}, 
now with $48=20+4\cdot 7$ and $68=20+4\cdot 12$ 
($d=4$ in place of $d=2$) 
and with $\Gamma={\rm O}^{+}(L_{cub})$ in place of $\tilde{{\rm O}}^{+}(L_{cub})$, 
we see that 
$\mathcal{F}_{7}$ has positive geometric genus 
and $\kappa(\mathcal{F}_{12})>0$. 
It remains to show that $\mathcal{F}_{5}$ is unirational. 
It suffices to prove 

\begin{proposition}
$M^{gtc}_{5}$ is unirational. 
\end{proposition}

\begin{proof}
We enlarge $M^{gtc}_{n}$ to the complete incidence over $|\mathcal{O}_{{\proj}^{5}}(3)|$: 
\begin{equation*}
(M^{gtc}_{n})^{\ast} = 
\{ \: (Y, C_{1}, \cdots, C_{n})\in |\mathcal{O}_{{\proj}^{5}}(3)|\times H^{n} \: | \: C_{i}\subset Y \: \}. 
\end{equation*}
The fiber of the projection 
$\pi\colon (M^{gtc}_{n})^{\ast}\to H^{n}$ over $(C_{1}, \cdots, C_{n})\in H^{n}$ 
is the linear system 
${\proj}V(C_{1},\cdots,C_{n})\subset |\mathcal{O}_{{\proj}^{5}}(3)|$ 
of cubics containing $C_{1}, \cdots, C_{n}$. 
When $n=5$, 
we have $\dim V(C_{1}, \cdots, C_{5})\geq 6$ for any $(C_{1}, \cdots, C_{5})\in H^{5}$, 
so $\pi$ is surjective, and 
there is a unique irreducible component of $(M^{gtc}_{5})^{\ast}$ 
of dimension $\geq 105$ that is birationally a ${\proj}^{N}$-bundle over $H^{5}$ with $N\geq 5$. 
On the other hand, 
$M^{gtc}_{5}$ is an open set of the unique irreducible component of $(M^{gtc}_{5})^{\ast}$ 
of dimension $105$ that dominates $|\mathcal{O}_{{\proj}^{5}}(3)|$. 
We want to show that these two irreducible components coincide: 
then $M^{gtc}_{5}\to H^{5}$ is dominant, 
and $M^{gtc}_{5}$ is birationally a ${\proj}^{5}$-bundle over $H^{5}$ 
and hence unirational. 

Let $(C_{1}, \cdots, C_{5})$ be a general point of $H^{5}$. 
By genericity we may assume that 
each $C_{i}$ is smooth and spans a $3$-plane $P_{i}\subset {\proj}^{5}$, 
$P_{i}\cap P_{j}$ is a line, and $C_{i}\cap P_{j}=\emptyset$. 
Let $(Y, C_{1}, \cdots, C_{5})$ be a general point of 
$\pi^{-1}(C_{1}, \cdots, C_{5})={\proj}V(C_{1}, \cdots, C_{5})$. 
It suffices to show that 
generalization of $(Y, C_{1}, \cdots, C_{5})$, 
i.e., small perturbation inside $(M^{gtc}_{5})^{\ast}$, 
contains $(Y', C_{1}', \cdots, C_{5}')$ with $Y'\in U$. 

We may assume that $Y$ is irreducible and contains no $3$-plane, 
because the locus of $(Y, C_{1}, \cdots, C_{5})$ with $Y$ reducible or containing a $3$-plane 
has dimension $<105$. 
Since each $C_{i}$ is smooth,  
the results of \cite{LLSS} \S 2 tell us that  
the cubic surface $S_{i}=Y\cap P_{i}$ is either 
(A) with at most ADE singularities or 
(B) integral but non-normal (singular along a line) or  
(C) reducible. 
By comparison of dimension again, 
we may assume that at least one, say $S_{1}$, is of type (A).  

Now $(C_{2}, \cdots, C_{5})$ is a general point of $H^{4}$. 
The projection $M^{gtc}_{4}\to H^{4}$ is dominant 
as can be checked similarly in an inductive way. 
Therefore there exists a cubic fourfold $Y''\in U$ containing $C_{2}, \cdots, C_{5}$. 
Let $Y'$ be a general member of the pencil $\langle Y, Y'' \rangle$.  
Since $Y''\in U$, we have $Y'\in U$. 
Since both $Y$ and $Y''$ contain $C_{2}, \cdots, C_{5}$, 
$Y'$ contains $C_{2}, \cdots, C_{5}$ too. 
In the fixed $3$-plane $P_{1}$, 
the cubic surface $S'=Y'\cap P_{1}$ degenerates to 
the cubic surface $S_{1}=Y\cap P_{1}$ with at most ADE singularities, 
so $S'$ has at most ADE singularities too. 
By \cite{LLSS} Theorem 2.1, 
the nets of twisted cubics on cubic surfaces degenerate flatly in such a family. 
Therefore we have a twisted cubic $C'\subset S'$ 
which specializes to $C_{1}\subset S_{1}$ as $Y'$ specializes to $Y$. 
Therefore $(Y', C', C_{2}, \cdots, C_{5})\in M_{5}^{gtc}$ specializes to $(Y, C_{1}, C_{2}, \cdots, C_{5})$. 
This proves our assertion. 
\end{proof}


\section{Varieties of power sums of cubic fourfolds}\label{sec:IR}

In this section we prove Theorem \ref{thm:main} for the case of Iliev-Ranestad fourfolds \cite{IR}. 
Let $H$ be the irreducible component of the Hilbert scheme 
${\rm Hilb}_{10}|\mathcal{O}_{{\proj}^{5}}(1)|$ 
of length $10$ subschemes of $|\mathcal{O}_{{\proj}^{5}}(1)|$ 
that contains the locus of $10$ distinct points. 
For a cubic fourfold $Y\subset {\proj}^{5}$ 
with defining equation $f\in H^{0}(\mathcal{O}_{{\proj}^{5}}(3))$, 
its variety of $10$ sums of powers 
$VSP(Y)=VSP(Y, 10)$ 
is defined as the closure in $H$ 
of the locus of distinct 
$([l_{1}], \cdots, [l_{10}])$ such that 
$f=\sum_{i}\lambda_{i}l_{i}^{3}$ 
for some $\lambda_{i}\in {\C}$. 
Iliev-Ranestad \cite{IR}, \cite{IR2} proved that 
when $Y$ is general, 
$VSP(Y)$ is a holomorphic symplectic manifold of $K3^{[2]}$ type, 
with polarization of Beauville norm $38$ and non-split type. 
(See also \cite{Mo} for the computation of polarization.) 
Hence its polarized Beauville lattice is isometric to 
\begin{equation*}\label{eqn:LDV} 
L_{IR} = 2U \oplus 2E_{8} \oplus K, \qquad 
K=\begin{pmatrix} -2 & 1 \\ 1 & -10 \end{pmatrix}. 
\end{equation*}
Let $\Gamma=\tilde{{\rm O}}^{+}(L_{IR})$. 

\begin{lemma}\label{lem:cusp form IR}
There exists an embedding $K\hookrightarrow E_{8}$ with $r(K^{\perp})=40$. 
The resulting cusp form $\Phi_{12}|_{L_{IR}}$ has weight $32$. 
Moreover, $S_{44}(\Gamma, \det)$ has dimension $\geq 2$. 
\end{lemma}

\begin{proof}
Let $v_{1}, v_{2}$ be the basis of $K$ in the above expression. 
We embed $K\hookrightarrow E_{8}$ by sending, in the model \eqref{eqn:E8} of $E_{8}$,  
\begin{equation*}
v_{1}\mapsto (1, -1, 0, \cdots, 0), \quad 
v_{2}\mapsto (0, 1, 3, 0, \cdots, 0). 
\end{equation*}
The roots of $E_{8}$ orthogonal to these two vectors are 
$\delta_{\pm i, \pm j}$ with $i, j\geq 4$, 
whose number is $2\cdot 5 \cdot 4=40$. 
Hence $\Phi_{12}|_{L_{IR}}$ has weight $12+20=32$. 
Furthermore, computing the formula in \cite{Br}, we see that 
$\dim M_{k}(\rho_{L_{IR}})=[(5k-3)/6]$. 
Product of $\Phi_{12}|_{L_{IR}}$ with the Gritsenko lift of $M_{3}(\rho_{L_{IR}})$ 
implies the last assertion. 
\end{proof}
 
Let $U$ be the open locus of $|\mathcal{O}_{{\proj}^{5}}(3)|$ 
where $VSP(Y)$ is smooth of dimension $4$ and $Y$ is smooth with no nontrivial stabilizer. 
The period map $U/{\rm PGL}_{6}\to {\GD}$  
is generically finite and dominant (\cite{IR}, \cite{IR2}). 
Consider the incidence 
\begin{equation*}
VSP_{n} = 
\{ \: (Y, \Gamma_{1}, \cdots, \Gamma_{n}) \in U \times H^{n} \: | \: \Gamma_{i}\in VSP(Y) \: \} 
\subset U \times H^{n} 
\end{equation*}
and let 
$\mathcal{F}_{n}=VSP_{n}/{\rm PGL}_{6}$. 
By Lemma \ref{lem:cusp form IR}, 
with $32=20+2\cdot 6$ and $44=20+2\cdot 12$,  
we see that $\mathcal{F}_{6}$ has positive geometric genus 
and $\kappa(\mathcal{F}_{12})>0$. 
On the other hand, 
as observed in \cite{IR}, 
$VSP_{1}$ is birationally a ${\proj}^{9}$-bundle over $H$ 
and hence rational. 
Therefore $\mathcal{F}_{1}$ is unirational. 
This proves Theorem \ref{thm:main} in the present case. 

\begin{remark}
There also exist embeddings $K\hookrightarrow E_{8}$ with $r(K^{\perp})=30$ 
(send $v_{2}$ to $(0, 1, 1, 2, 2, 0, 0, 0)$ or to $(0, 1, 1, 1, 1, 1, 1, 2)$), 
but the resulting cusp form has weight $27$, 
which is not of the form $20+2n$. 
This, however, suggests that $\kappa\geq 0$ would actually start at least from $n=4$. 
\end{remark}


\section{Double EPW series}\label{sec:EPW}

In this section we prove Theorem \ref{thm:main} for 
the cases of double EPW sextics by O'Grady \cite{OG} 
and of double EPW cubes by Iliev-Kapustka-Kapustka-Ranestad \cite{IKKR}. 
They share some common features: 
both are parametrized by the Lagrangian Grassmannian $LG=LG(\bigwedge^{3}{\C}^{6})$, 
where $\bigwedge^{3}{\C}^{6}$ is equipped with the canonical symplectic form 
$\bigwedge^{3}{\C}^{6} \times \bigwedge^{3}{\C}^{6} \to \bigwedge^{6}{\C}^{6}$. 
Both are constructed as double covers of degeneracy loci related to $\bigwedge^{3}{\C}^{6}$. 
And both have $L_{EPW}=2U\oplus 2E_{8}\oplus 2A_{1}$ 
as the polarized Beauville lattices. 
Thus they share the same parameter space and essentially the same period space. 

The presence of covering involution requires extra care 
in the construction of the universal (or perhaps we should say rather ``tautological'') family 
over a Zariski open set of the moduli space.

\subsection{Double EPW sextics}\label{ssec:OG}

We recall the construction of double EPW sextics following \cite{OG}, \cite{OG2}. 
Let $F$ be the vector bundle over ${\proj}^{5}$ 
whose fiber over $[v]\in {\proj}^{5}$ is the image of 
${\C}v \wedge (\bigwedge^{2}{\C}^{6})\to \bigwedge^{3}{\C}^{6}$. 
For $[A]\in LG$ we write 
$Y_{A}[k]\subset {\proj}^{5}$ 
for the locus of those $[v]\in {\proj}^{5}$ such that $\dim (A\cap F_{v})\geq k$. 
We say that $A$ is generic if 
$Y_{A}[3]= \emptyset$ and 
${\proj}A \cap G(3, 6) = \emptyset$ in ${\proj}(\bigwedge^{3}{\C}^{6})$. 
In that case, 
$Y_{A}=Y_{A}[1]$ is a sextic hypersurface in ${\proj}^{5}$ singular along $Y_{A}[2]$, 
$Y_{A}[2]$ is a smooth surface, and 
$Y_{A}$ has a transversal family of $A_{1}$-singularities along $Y_{A}[2]$. 
Let $\lambda_{A}\colon F\to (\bigwedge^{3}{\C}^{6}/A) \otimes \mathcal{O}_{{\proj}^{5}}$ 
be the composition of the inclusion 
$F\hookrightarrow \bigwedge^{3}{\C}^{6} \otimes \mathcal{O}_{{\proj}^{5}}$ 
and the projection 
$\bigwedge^{3}{\C}^{6}\otimes \mathcal{O}_{{\proj}^{5}} \to 
(\bigwedge^{3}{\C}^{6}/A) \otimes \mathcal{O}_{{\proj}^{5}}$. 
Then ${\rm coker}(\lambda_{A})=i_{\ast}\zeta_{A}$ 
for a coherent sheaf $\zeta_{A}$ on $Y_{A}$ 
where $i:Y_{A}\hookrightarrow {\proj}^{5}$ is the inclusion. 
Let $\xi_{A}=\zeta_{A}\otimes \mathcal{O}_{Y_{A}}(-3)$. 
If we choose a Lagrangian subspace $B$ of $\bigwedge^{3}{\C}^{6}$ transverse to $A$, 
one can define a multiplication $\xi_{A}\times \xi_{A}\to \mathcal{O}_{Y_{A}}$. 
Although $B$ is necessary for the construction, 
the resulting multiplication does not depend on the choice of $B$ (\cite{OG2} p.152). 
Then let 
$X_{A}={\rm Spec}(\mathcal{O}_{Y_{A}}\oplus \xi_{A})$. 
This is a double cover of $Y_{A}$. 
If $A$ is generic in the above sense, 
$X_{A}$ is a holomorphic symplectic manifold of $K3^{[2]}$ type. 
The polarization (pullback of $\mathcal{O}_{{\proj}^{5}}(1)$) 
has Beauville norm $2$ and is of split type, 
and the polarized Beauville lattice is isometric to $L_{EPW}$. 
If $LG^{\circ}\subset LG$ is the open locus of generic $A$, 
the period map 
$LG^{\circ}/{\rm PGL}_{6}\to {\GD}$, 
where $\Gamma=\tilde{{\rm O}}^{+}(L_{EPW})$, 
is birational (\cite{OG} \S 6 and \cite{Mark} \S 8). 

\begin{lemma}\label{lem:cusp form EPW6}
The cusp form $\Phi_{12}|_{L_{EPW}}$ has weight $42$. 
Moreover, $S_{58}(\Gamma, \det)$ has dimension $\geq 2$. 
\end{lemma}

\begin{proof}
We embed $2A_{1}$ in $E_{8}$ in any natural way. 
Then $(2A_{1})^{\perp}\simeq D_{6}$ has $60$ roots,  
so $\Phi_{12}|_{L_{EPW}}$ has weight $42$. 
Working out the formula in \cite{Br}, we see that 
$\dim M_{k}(\rho_{L_{EPW}})=[k/3]$. 
Product of $\Phi_{12}|_{L_{EPW}}$ with the Gritsenko lift of $M_{7}(\rho_{L_{EPW}})$ 
implies the second assertion. 
\end{proof}

The construction of double EPW sextics can be done 
over a Zariski open set of the moduli space as follows (cf.~\cite{OG2}). 
Let $LG'\subset LG^{\circ}$ be the open locus where 
$A$ has no nontrivial stabilizer, 
and 
$\pi_{1}\colon LG'\times {\proj}^{5}\to LG'$, 
$\pi_{2}\colon LG'\times {\proj}^{5}\to {\proj}^{5}$ 
be the projections. 
Let 
$Y=\cup_{A}Y_{A}\subset LG'\times {\proj}^{5}$ 
be the universal family of EPW sextics over $LG'$. 

\begin{lemma}\label{lem:shrink} 
There exists a ${\rm PGL}_{6}$-invariant Zariski open set $LG''$ of $LG'$ such that 
$\mathcal{O}_{LG''\times {\proj}^{5}}(Y) \simeq \pi_{2}^{\ast}\mathcal{O}_{{\proj}^{5}}(6)$ 
as ${\rm PGL}_{6}$-linearized line bundles over $LG''\times {\proj}^{5}$. 
\end{lemma} 

\begin{proof}
Consider the quotient 
$\mathcal{Y}=Y/{\rm PGL}_{6}$, 
which is a divisor of the Brauer-Severi variety 
$\mathcal{P}=(LG'\times{\proj}^{5})/{\rm PGL}_{6}$ 
over $\mathcal{M}=LG'/{\rm PGL}_{6}$. 
Each fiber of $\mathcal{Y}\to \mathcal{M}$ is a canonical divisor of 
the fiber of $\pi:\mathcal{P}\to\mathcal{M}$. 
This implies that  
$\mathcal{O}_{\mathcal{P}}(\mathcal{Y}) \simeq 
K_{\pi}\otimes \pi^{\ast}\mathcal{O}_{\mathcal{M}}(D)$ 
for some divisor $D$ of $\mathcal{M}$. 
Removing the support of $D$ from $\mathcal{M}$, 
we obtain $\mathcal{O}_{\mathcal{P}}(\mathcal{Y}) \simeq K_{\pi}$ 
over its complement. 
Pulling back this isomorphism to $LG'\times {\proj}^{5}\to LG'$, 
we obtain the desired ${\rm PGL}_{6}$-equivariant isomorphism. 
\end{proof}

We rewrite $LG''=LG'$ and $Y|_{LG''}=Y$. 
Let $E$ be the universal quotient vector bundle of rank $10$ over $LG'$. 
We have a natural homomorphism 
$\lambda\colon \pi_{2}^{\ast}F\to \pi_{1}^{\ast}E$ 
over $LG'\times {\proj}^{5}$ 
whose restriction to $\{ A \} \times {\proj}^{5}$ is $\lambda_{A}$, 
and ${\rm coker}(\lambda)=i_{\ast}\zeta$ 
for a coherent sheaf $\zeta$ on $Y$ 
where $i\colon Y\to LG'\times {\proj}^{5}$ is the inclusion. 
As was done in \cite{OG2}, 
if we choose $B\in LG$ and let 
$U_{B}\subset LG'$ be the open locus of those $A$ transverse to $B$, 
we have a multiplication $\zeta\times \zeta \to \mathcal{O}_{Y}(Y)$ 
over $Y|_{U_{B}}$. 
Since the multiplication does not depend on the choice of $B$ at each fiber, 
we obtain an ${\rm SL}_{6}$-equivariant multiplication 
$\zeta\times \zeta \to \mathcal{O}_{Y}(Y)$ 
over the whole $Y$. 
If we put 
$\xi=\zeta\otimes \mathcal{O}_{Y}(-3)$, 
Lemma \ref{lem:shrink} enables us to pass to an ${\rm SL}_{6}$-equivariant multiplication 
$\xi\times \xi \to \mathcal{O}_{Y}$. 
Since the scalar matrices in ${\rm SL}_{6}$ act trivially on $\xi$, 
$\xi$ is actually ${\rm PGL}_{6}$-linearized and 
this multiplication is ${\rm PGL}_{6}$-equivariant. 

Now taking $X={\rm Spec}(\mathcal{O}_{Y}\oplus \xi)$, 
we obtain a universal family of double EPW sextics over $LG'$ 
acted on by ${\rm PGL}_{6}$. 
Let
$\mathcal{M}=LG'/{\rm PGL}_{6}$, 
$\mathcal{F}=X/{\rm PGL}_{6}$ and 
$\mathcal{F}_{n}=\mathcal{F}\times_{\mathcal{M}}\cdots \times_{\mathcal{M}} \mathcal{F}$ 
($n$ times). 
Note that this is not a moduli space even birationally, 
as it is not mod out by the covering involution. 
By Lemma \ref{lem:cusp form EPW6}, 
with $42=20+2\cdot 11$ and $58=20+2\cdot 19$, 
we see that 
$\mathcal{F}_{11}$ has positive geometric genus and 
$\kappa(\mathcal{F}_{19})>0$. 
This proves Theorem \ref{thm:main} in the case of double EPW sextics.

\subsection{Double EPW cubes}\label{ssec:IKKR}

We recall the construction of double EPW cubes following \cite{IKKR}. 
For $[U]\in G(3, 6)$, we write 
$T_{U}=(\bigwedge^{2}U)\wedge {\C}^{6} \subset \bigwedge^{3}{\C}^{6}$. 
For $[A]\in LG$, let 
$D_{k}^{A}\subset G(3, 6)$ be the locus of those $[U]$ with $\dim (A\cap T_{U})\geq k$. 
We say that $A$ is generic if 
$D^{A}_{4}=\emptyset$ and 
${\proj}A \cap G(3, 6)=\emptyset$ in ${\proj}(\bigwedge^{3}{\C}^{6})$. 
In that case, 
$D^{A}_{2}$ is a sixfold singular along $D^{A}_{3}$, 
$D^{A}_{3}$ is a smooth threefold, and 
the singularities of $D^{A}_{2}$ is a transversal family of 
$\frac{1}{2}(1, 1, 1)$ quotient singularity along $D^{A}_{3}$. 
Let $\tilde{D}_{2}^{A}\to D^{A}_{2}$ be the blow-up at $D^{A}_{3}$ 
and $E\subset \tilde{D}_{2}^{A}$ be the exceptional divisor. 
Then $\tilde{D}_{2}^{A}$ is smooth, and 
$E$ is a smooth bi-canonical divisor of $\tilde{D}_{2}^{A}$ (\cite{IKKR} p.~254). 
Take the double cover $\tilde{Y}_{A}\to \tilde{D}_{2}^{A}$ branched over $E$ and 
contract the ${\proj}^{2}$-ruling of the ramification divisor 
by using pullback of some multiple of $\mathcal{O}_{D^{A}_{2}}(1)$. 
This produces a holomorphic symplectic manifold $Y_{A}$ 
of $K3^{[3]}$ type (\cite{IKKR} Theorem 1.1). 

The polarization has Beauville norm $4$ and divisibility $2$,  
so the polarized Beauville lattice is isometric to $L_{EPW}$ by \cite{GHS10}. 
The monodromy group is evidently contained in ${\rm O}^{+}(L_{EPW})$ 
(but whether it is smaller seems unclear to me). 
The quotient ${\rm O}^{+}(L_{EPW})/\tilde{{\rm O}}^{+}(L_{EPW})$ 
is $\frak{S}_{2}$ generated by the switch of the two copies of $A_{1}$, 
say $\iota\in {\rm O}^{+}(L_{EPW})$. 
Construction of cusp forms becomes more delicate than the previous cases, 
as $\Phi_{12}|_{L_{EPW}}$ is \textit{anti}-invariant under $\iota$. 

\begin{lemma}\label{lem:cusp form EPW cube}
Let $\Gamma={\rm O}^{+}(L_{EPW})$. 
Then  
$S_{68}(\Gamma, \det) \ne \{ 0 \}$ and 
$S_{80}(\Gamma, \det)$ has dimension $\geq 2$. 
\end{lemma}

\begin{proof}
We abbreviate $L=L_{EPW}$. 
We first verify that $\Phi_{12}|_{L}$ is $\iota$-anti-invariant. 
Let $\iota'$ be the involution of the $D_{6}$ lattice induced by 
the involution of its Dynkin diagram. 
Then $\iota \oplus \iota'$ extends to an involution $\tilde{\iota}$ of $II_{2,26}$. 
The modular form $\Phi_{12}$ is $\tilde{\iota}$-invariant. 
If we run $\delta$ over the positive roots of $D_{6}$, 
the product $\prod_{\delta}(\delta, \cdot)$ is also $\tilde{\iota}$-invariant 
because $\iota'$ permutes the positive roots of $D_{6}$. 
Therefore $\Phi_{12}/\prod_{\delta}(\delta, \cdot )$ 
as a section of $\mathcal{L}^{\otimes 42}\otimes \det$ 
over $\mathcal{D}_{II_{2,26}}$ is $\tilde{\iota}$-invariant. 
Since $\det(\tilde{\iota})=1$ while $\det(\iota)=-1$, 
this shows that 
$\Phi_{12}|_{L}$ as a section of $\mathcal{L}^{\otimes 42}\otimes \det$ over $\mathcal{D}_{L}$
is anti-invariant under $\iota$. 

In order to construct $\iota$-invariant cusp forms of character $\det$, 
we take product of $\Phi_{12}|_{L}$ with 
the Gritsenko lift of the $\iota$-anti-invariant part of $M_{k}(\rho_{L})$.  
By the formulae in \cite{Br} and \cite{Ma2}, 
we see that 
$\dim M_{k}(\rho_{L})=[k/3]$ and 
$\dim M_{k}(\rho_{L})^{\iota} = [(k+2)/4]$ 
for $k>2$ odd. 
We also require the congruence condition  
$42+k+9 \equiv 20$ mod $3$, namely $k\equiv 2$ mod $3$. 
Now, when $k=17$ (resp.~$k=29$), 
the $\iota$-anti-invariant part has dimension $1$ (resp.~$2$). 
This proves our claim. 
\end{proof}

We can do the double cover construction over a Zariski open set of the moduli space. 
Let $LG^{\circ}\subset LG$ be the open set of generic $[A]$ 
which is ${\rm PGL}_{6}$-stable and has no nontrivial stabilizer. 
Let $D_{2}=\cup_{A}D_{2}^{A} \subset LG^{\circ}\times G(3, 6)$ 
be the universal family of $D^{A}_{2}$'s. 
We have the geometric quotients 
$\mathcal{M}=LG^{\circ}/{\rm PGL}_{6}$, 
$\mathcal{Z}=D_{2}/{\rm PGL}_{6}$ 
with projection $\mathcal{Z}\to \mathcal{M}$. 
The relative $\mathcal{O}(2)$ descends. 
Let $\tilde{\mathcal{Z}}\to \mathcal{Z}$ be the blow-up at ${\rm Sing}(\mathcal{Z})$, 
$\mathcal{B}\subset \tilde{\mathcal{Z}}$ be the exceptional divisor, 
and $\pi\colon \tilde{\mathcal{Z}}\to \mathcal{M}$ be the projection. 
As in the proof of Lemma \ref{lem:shrink}, 
we may shrink $\mathcal{M}$ to a Zariski open set $\mathcal{M}'\subset \mathcal{M}$ 
so that $\mathcal{B}|_{\mathcal{M}'}\sim 2K_{\pi}$. 
Then we can take the double cover of 
$\tilde{\mathcal{Z}}|_{\mathcal{M}'}$ branched over $\mathcal{B}|_{\mathcal{M}'}$. 
Contracting the ramification divisor relatively by using 
pullback of a multiple of the relative $\mathcal{O}(2)$, 
we obtain a universal family $\mathcal{F}\to \mathcal{M}'$ 
of double EPW cubes over $\mathcal{M}'$. 
Then let 
$\mathcal{F}_{n}=\mathcal{F}\times_{\mathcal{M}'} \cdots \times_{\mathcal{M}'} \mathcal{F}$ 
($n$ times). 

The period map $\mathcal{M}\to {\GD}$ is generically finite and dominant (\cite{IKKR} Proposition 5.1). 
By Lemma \ref{lem:cusp form EPW cube}, 
with $68=20+3\cdot 16$ and $80=20+3\cdot 20$, 
we see that $\mathcal{F}_{16}$ has positive geometric genus 
and $\kappa(\mathcal{F}_{20})>0$. 
This proves Theorem \ref{thm:main} in the case of double EPW cubes. 



\end{document}